\documentclass[a4paper]{amsart}
\usepackage{amssymb}
\usepackage{graphicx}
\usepackage{url}
\usepackage{hyperref}

\newcommand{\dt}{{\Delta t}}
\newcommand{\Dt}{{\Delta t}}

\newcommand{\R}{\mathbb R}
\newcommand{\N}{\mathbb N}

\newcommand{\norm}[1]{\left\Vert#1\right\Vert}

\newcommand{\seq}[1]{\left\{#1\right\}}

\newcommand{\pl}{\partial^\ell_x}

\newcommand{\Curl}{\operatorname{curl}}

\newtheorem{theorem}{Theorem}[section]

\newtheorem{lemma}[theorem]{Lemma}

\numberwithin{equation}{section}   

\allowdisplaybreaks

\begin{document}
\title[Operator splitting for PDEs with Burgers nonlinearity]{Operator
  splitting for partial differential equations with Burgers
  nonlinearity}

\author[Holden]{Helge Holden} \address[Holden]{\newline Department of
  Mathematical Sciences, Norwegian University of Science and
  Technology, NO--7491 Trondheim, Norway,\newline {\rm and} \newline
  Centre of Mathematics for Applications,
  University of Oslo, P.O.\ Box 1053, Blindern, NO--0316 Oslo, Norway} 
\email[]{\href{holden@math.ntnu.no}{holden@math.ntnu.no}}
\urladdr{\href{http://www.math.ntnu.no/~holden}{www.math.ntnu.no/\~{}holden}}

\author[Lubich]{Christian Lubich} \address[Lubich]{\newline
  Mathematisches Institut, Universit\"at T\"ubingen, Auf der
  Morgenstelle 10, D--72076 T\"ubingen, Germany}
\email[]{\href{lubich@na.uni-tubingen.de}{lubich@na.uni-tubingen.de}}
\urladdr{\href{http://na.uni-tuebingen.de/~lubich/}{http://na.uni-tuebingen.de/\~{}lubich/}}

\author[Risebro]{Nils Henrik Risebro} \address[Risebro]{\newline
  Centre of Mathematics for Applications,
  University of Oslo, P.O.\ Box 1053, Blindern, NO--0316 Oslo, Norway}
\email[]{\href{nilshr@math.uio.no}{nilshr@math.uio.no}}
\urladdr{\href{folk.uio.no/nilshr}{folk.uio.no/nilshr}}

\date{\today}



\thanks{Supported in part by the Research Council of Norway.}

\begin{abstract}
  We provide a new analytical approach to operator splitting for
  equations of the type $u_t=Au+u u_x$ where $A$ is a linear
  differential operator such that the equation is well-posed.
  Particular examples include the viscous Burgers' equation, the
  Korteweg--de Vries (KdV) equation, the Benney--Lin equation, and the
  Kawahara equation. We show that the Strang splitting method
  converges with the expected rate if the initial data are
  sufficiently regular. In particular, for the KdV equation we obtain
  second-order convergence in $H^r$ for initial data in $H^{r+5}$ with
  arbitrary $r\ge 1$.
\end{abstract}

\maketitle
\section{Introduction} \label{sec:intro}

In this paper we study initial value problems of the type
\begin{equation}
  \label{pde}
  u_t=P(\partial_x)u+uu_x \quad (x\in\R, \  0\le t \le T), \quad u|_{t=t_0}=u_0,
\end{equation}
with a polynomial $P$ of degree $\ell\ge 2$ satisfying
\begin{equation}
  \label{P}
  \hbox{Re}\, P(i\xi) \le 0 \quad\hbox{ for all  $\xi\in\R$.}
\end{equation}
This class includes several important equations like
\begin{align*}
  u_t&=u_{xx}+ uu_x, && \text{the viscous Burgers equation},\\
  u_t&=u_{xxx}+ uu_x,&& \text{the Korteweg--de Vries equation,}\\
  u_t&=-u_{xxx}-\beta(u_{xx}+u_{xxxx})-u_{xxxxx}+uu_x, &&\text{the Benney--Lin equation \cite{Benney:1966,Lin:1974},}\\
  u_t&=u_{xxxxx}-u_{xxx}+ uu_x, && \text{the Kawahara
    equation \cite{Kawahara:1975}. }\\
\end{align*}
We employ operator splitting (see, e.g.,
\cite{HoldenKarlsenLieRisebro:2009}), i.e., construction of an
approximate solution by concatenating the solutions of the separate
problems
\begin{equation}
  \label{eq: burgers}
  v_t=P(\partial_x)v\quad\hbox{ and }\quad
  w_t=ww_x. 
\end{equation}
More precisely, if we let $u(t)=\Phi_C^t(u_0)$ denote the solution of
the initial value problem
\begin{equation}
  \label{eq:intro2}
  u_t=C(u), \quad u|_{t=0}=u_0,
\end{equation}
then (sequential) operator splitting claims that the exact solution
$u(t)=\Phi_C^t(u_0)$ is well approximated by $u_n$, at $t= n\Dt\le T$
and as $\Dt\to 0$, where
\begin{equation}
  \label{lie-split}
  u_{n+1}=\Phi_A^{\Dt}(\Phi_B^{\Dt}(u_n)), \quad n=0,1,2,\ldots, 
\end{equation}
and $C=A+B$. In our case we have the operator 
\begin{equation}
  \label{air}
Au=P(\partial_x)u
\end{equation}
 and
$B$ will be the Burgers operator 
\begin{equation}
  \label{burgers}
B(u)=u u_x
\end{equation}
acting on appropriate Sobolev spaces.  In the present paper a more
refined operator splitting, known as Strang splitting, will be
discussed. Strang splitting means that we consider the approximation
\begin{equation}
  \label{strang}
  u_{n+1}=\Psi^{\Dt}(u_n)=\Phi_A^{\Dt/2}\circ \Phi_B^{\Dt}\circ\Phi_A^{\Dt/2}(u_n),
  \quad n=0,1,2,\ldots\, .
\end{equation}
The technique we use requires a well-posedness theory for the full
equation in Sobolev spaces. This is currently available for the
equations listed above, see \cite{Lady:1968} (the viscous Burgers
equation), \cite{BonaSmith:1975,KenigPonceVega:1991,LinaresPonce:2009}
(the KdV equation), \cite{BiagioniLinares:1997} (the Benney--Lin
equation), and \cite{JiaHuo:2009} (the Kawahara equation). 
The challenge in
using operator splitting for the present class of equations is the
fact that the Burgers step, i.e., solving the second equation in
\eqref{eq: burgers}, generically introduces singularities in finite
time irrespective of the smoothness of the initial data, while the
solution of the full problem \eqref{eq:intro2} remains smooth. Thus
the determination of the time step used in the operator splitting and
the control of the smoothness of the approximate solution are rather
delicate. This implies that the standard estimates that are required  for operator splitting, see, e.g., \cite{HoldenKarlsenLieRisebro:2009}, do not apply. The idea of applying operator splitting to the KdV equation was introduced in a short note by Tappert  \cite{Tappert:1974}, with numerical tests.  A further study was done in \cite{HoldenKarlsenRisebro:1999} where a Lax--Wendroff theorem was proved, namely assuming convergence of operator splitting, one can prove that the limit is a weak solution of the KdV equation. Furthermore, extensive numerical computations were  presented, all indicating that operator splitting does converge. A completely different and novel approach was taken in the paper \cite{HoldenKarlsenRisebroTao:2009}. Previous results had as goal to use operator splitting to show existence of solutions of the KdV equation.   However, in \cite{HoldenKarlsenRisebroTao:2009} one took for granted the well-posedness of the Cauchy problem for the KdV equation in Sobolev spaces of high order.  By using a bootstrap principle (see \cite{Tao:2006}), one could show that the Burgers step in a Sobolev space of lower order would avoid blow-up, thereby securing that the full approximation remained smooth, and indeed converged to the solution of the KdV equation. A key part of this analysis was to secure a uniform time step in the approximation. The approach in \cite{HoldenKarlsenRisebroTao:2009}  also presented a new interpolation between the discrete time steps based on the introduction of a new auxiliary time variable, and hence a two-variable approximation was used.

In the present paper we take a different approach. We do take as a starting point the well-posedness of the Cauchy problem for the KdV equation, and we use the same analysis to estimate carefully the change in Sobolev norm during the Burgers step. However, instead of introducing a two-variable approximation and applying the bootstrap principle, we perform an analysis that identifies the principal error terms of the local error as quadrature errors, which are estimated via bounds of Lie commutators. Such a type of analysis was first done for operator splitting for linear evolution equations in \cite{JahnkeLubich:2000} and for non-linear Schr\"odinger equations in \cite{Lubich:2008}.  The error propagation is studied via ``Lady Windermere's fan'' \cite{HairerNorsettWanner:1993}, as is usual in the error analysis of time discretization methods, controlling here carefully the spatial regularity of the numerical approximations by first establishing lower-order convergence in a higher-order Sobolev norm, as in \cite{Lubich:2008}. In this way we improve and simplify the results in  
\cite{HoldenKarlsenRisebroTao:2009}.   The main result says that the Strang splitting converges to the full solution with the expected second-order convergence rate, in the $H^r$ Sobolev norm for $r\ge 1$ in the case of $H^{r+5}$ initial data for the KdV equation. Furthermore, as was clear already in  \cite{HoldenKarlsenRisebroTao:2009}, the analysis extends to a class of dispersive equations, and here we find it suitable to study the class \eqref{pde} subject to the condition \eqref{P}. This means that the equations all share a Burgers nonlinearity. The extension of the present approach to dispersive equations with other nonlinearities, not necessarily quadratic and also in higher spatial dimensions, appears feasible as long as an analogue of the regularity properties of the Burgers step can be established.  We also refer to \cite{HoldenKarlsenKarper:2011} where the approach of 
\cite{HoldenKarlsenRisebroTao:2009} is applied to the generalized active scalar equation 
\begin{equation*}
	\theta_t + \boldsymbol{u} \cdot \nabla \theta 
	+ \Lambda^\alpha \theta = 0, 
	\qquad \boldsymbol{u} = \Curl \Lambda^{-\beta}\theta,
	\end{equation*}
where  $\nabla$ is the spatial gradient operator and $\Lambda=(-\Delta)^{1/2}$. The unknown function is $\theta\colon [0,T]\times \R^2\to \R$ and we assume $\alpha \in (0,2]$.
  
\section{Error bounds for Strang splitting: statement of results}
Let $r$ be a positive integer and associate with it
$$
p=r+2\ell-1, \qquad q=r+\ell-1= p-\ell,
$$
where $\ell\ge 2$ is the degree of the polynomial $P$ in
\eqref{pde}. The integers $p$, $q$ and $r$ will be kept fixed throughout the
paper.

We require local \emph{well-posedness of \eqref{pde} in $H^{r}$ and
  $H^q$} in the following sense: For both $s=r$ and $s=q$ and for a
given time $T$, there exists $R>0$ such for all $u_0\in H^s$ with $\|
u_0\|_{H^s}\le R$, there is a unique strong solution $u\in C([0,T],
H^s)$ of \eqref{pde} with initial data $u_0$, and the dependence on
the initial data is locally Lipschitz-continuous: There is a constant 
$K=K(R,T)<\infty$ such that the solutions $\widetilde u$, $u$
corresponding to arbitrary initial data $\widetilde u_0$, $u_0$ in the
$H^s$-ball of radius $R$ satisfy
\begin{equation}\label{lip-flow}
  \| \widetilde u(t) - u(t) \|_{H^s} \le K \| \widetilde u_0 - u_0 \|_{H^s} 
  \quad\hbox{ for } \  0\le t \le T.
\end{equation}

For the KdV equation, it is known from \cite{KenigPonceVega:1991} that
this well-posedness assumption holds for every $s\in \N$, for
arbitrary final time $T$ and with arbitrary positive~$R$.

For the general class \eqref{pde},\eqref{P}, it can be shown with
arguments as in Lemma~\ref{lem:burgers-2} below, that the above
well-posedness result holds, at least, with sufficiently small
$R=R(T)$ for any given $T$, or for sufficiently small $T=T(R)$ for any
given $R$, the same as for the inviscid Burgers equation. For the
equations listed in the introduction, however, the well-posedness
situation is more favorable than for the inviscid Burgers equation;
see \cite{Lady:1968} (the viscous Burgers equation),
\cite{BonaSmith:1975,KenigPonceVega:1991,LinaresPonce:2009} (the KdV
equation), \cite{BiagioniLinares:1997} (the Benney--Lin equation), and
\cite{JiaHuo:2009} (the Kawahara equation).

We consider solutions bounded by 
\begin{equation}\label{HYP}
\| u(t) \|_{H^p} \le \rho <R \quad\hbox{ for } \ 0\le t \le T.
\end{equation}
In particular we assume that the initial data are in $H^p$:
$$
u_0\in H^p.
$$
Under these assumptions we will in this paper show the following results for the
Strang splitting \eqref{strang}.

\begin{theorem}
  \label{thm:one} {\bf (First-order convergence in $H^q$ and bound in
    $H^p$)}
  There is $\overline{\dt}>0$ such that for $\dt\le \overline{\dt}$
  and $t_n=n\dt\le T$,
$$
\| u_n - u(t_n) \|_{H^q} \le C_1 \dt \qquad\hbox{and}\qquad \| u_n
\|_{H^p} \le C_0.
$$
Here, $\overline{\dt}$, $C_0$ and $C_1$ only depend on
$\|u_0\|_{H^p}$, $\rho$, and $T$.
\end{theorem}

This result will be used to prove our main result:
\begin{theorem}
  \label{thm:two} {\bf (Second-order convergence in $H^{r}$)}  Assume
  that we have a solution $u$ of \eqref{pde} that satisfies
  \eqref{HYP}.  Define the Strang approximation $u_n$ by
  \eqref{strang}.  Then the following statement holds: 
  There is $\overline{\dt}>0$ such that for $\dt\le \overline{\dt}$
  and $t_n=n\dt\le T$,
$$
\| u_n - u(t_n) \|_{H^{r}} \le C_2 \dt^2.
$$
Here, $\overline{\dt}$ and $C_2$ only depend on $\|u_0\|_{H^p}$,
$\rho$, and $T$.
\end{theorem}
For the KdV equation, this result yields second-order convergence in
$H^r$ for initial data in $H^{r+5}$ for arbitrary $r\in \N$.  This
improves significantly on the result of
\cite{HoldenKarlsenRisebroTao:2009}, where second-order convergence is
shown in the $H^r$-norm for $r\ge 8$ for initial data in
$H^{r+9}$. The proof uses a regularity result for the inviscid Burgers
equation as in \cite{HoldenKarlsenRisebroTao:2009}, but is otherwise
quite different from the proof given there. It is conceptually closely
related to the error analysis of Strang splitting given in
\cite{Lubich:2008} for the nonlinear Schr\"odinger equation. The proofs
of Theorems~\ref{thm:one} and~\ref{thm:two} will cover the remainder
of the paper. We remark that throughout our computations, $C$ denotes
a generic constant whose value may change at each occurrence.

\section{Regularity results for the inviscid Burgers equation}
The following variant of a result in
\cite{HoldenKarlsenRisebroTao:2009} will play a key role in the proof
of Theorem~\ref{thm:one}.
\begin{lemma}
  \label{lem:burgers-1}
  If $\|\Phi_B^t(u_0) \|_{H^q} \le \alpha$ for $0\le t\le\dt$, then
  $\|\Phi_B^t(u_0) \|_{H^p} \le e^{c\alpha t}\| u_0 \|_{H^p}$ for
  $0\le t\le\dt$, where $c$ is independent of $u_0$ and $\dt$.
\end{lemma}
\begin{proof}
  We find that $w(t)=\Phi_{B}^t(u_0)$ satisfies
  (cf.~\cite{HoldenKarlsenRisebroTao:2009})
  \begin{equation}
    \label{eq:kdv4}
    \begin{aligned}
      &\| w\|_{H^p}\,\frac{d}{dt}\| w\|_{H^p}=
      \frac12\frac{d}{dt}\norm{\Phi_{B}^t(u_0)}_{H^p}^2=(w,w_t)_{H^p}
      \\
      &= \sum_{j=0}^p \int_{\R} \partial^j_x w\, \partial^j_x (w
      w_x)\, dx = \sum_{j=0}^p \sum_{k=0}^j \binom{j}{k} \int
      \partial^j_x w \, \partial^{k+1}_xw\, \partial^{j-k}_xw \,dx.
    \end{aligned}
  \end{equation}
  For $j<p$, any of the above terms can be estimated by
  \begin{align*}
    \Bigl| \int \partial^j_x w\,\partial^{k+1}_x w\, \partial^{j-k}_x
    w\,dx\Bigr|&\le \norm{\partial^{j}_x w}_{L^\infty}
    \norm{\partial^{\max\seq{k+1,j-k}}_x w}_{L^2}
    \norm{\partial^{\min\seq{k+1,j-k}}_x w}_{L^2}
    \\
    &\le C \norm{w}_{H^p}^2 \norm{w}_{H^{q}} \\
    &\le C \alpha \norm{w}_{H^p}^2,
  \end{align*}
  by using the Sobolev inequality $\norm{v}_{L^\infty}\le \frac1{\sqrt{2}}
  \norm{v}_{H^1}$ (so that $C=1/\sqrt{2}$) and the fact that
  $\min\seq{k+1,j-k}\le \tfrac12(j+1) \le \tfrac12 p \le p-\ell=q$ since
  $p\ge 2\ell$.  For
  $j=p$ we estimate for $k\le q$ and $k\ne q-1$ by
  \begin{align*}
    I:= \Bigl| \int \partial^p_x w\,\partial^{k+1}_x
    w\, \partial^{p-k}_x w\,dx\Bigr| 
    &\le \norm{\partial^{p}_x
      w}_{L^2}\norm{\partial^{k+1}_x w}_{L^\infty}
    \norm{\partial^{p-k}_x w}_{L^2}
    \\
    &\le C \norm{w}_{H^p}\norm{w}_{H^{k+2}} \norm{ w}_{H^{p-k}}.
  \end{align*}
  To estimate this further we have to distinguish the cases  $k+2\le q$  and $k=q$. In the former case  we find
$$
I \le C \norm{w}_{H^p} \norm{w}_{H^q}\norm{w}_{H^{p}} \le C \alpha
\norm{w}_{H^p}^2,
$$
while in the latter case we obtain
$$
I\le C \norm{w}_{H^p}\norm{w}_{H^{p}}\norm{w}_{H^{q}} \le C \alpha
\norm{w}_{H^p}^2,
$$
since $q+2\le q+\ell=p$ and $p-q=\ell\le \ell +r-1 =q$.

In the two cases where either $k+1=q$ or  $q+2\le k+1 \le p$ we estimate 
\begin{align*}
  I &\le \norm{\partial^{p}_x w}_{L^2}\norm{\partial^{k+1}_x w}_{L^2}
  \norm{\partial^{p-k}_x w}_{L^\infty}
  \\
  &\le C\norm{w}_{H^p}\norm{w}_{H^{k+1}} \norm{ w}_{H^{p-k+1}}.
\end{align*}
In both cases this is bounded by $C \alpha \norm{w}_{H^p}^2$, since in
the case $k+1=q$ we have $p-k+1=\ell+2\le 2\ell \le p$, and for $k+1\ge
q+2$ we have $p-k+1\le p-q = \ell \le q$.

We are left with the term where $k=p=j$, viz.
\begin{align*}
  \Bigl| \int \partial^p_x w\, \partial^{p+1}_x w \, w \,dx \Bigr| &=
  \frac{1}{2}\Bigl | \int \left(\partial^p_x w\right)^2 \, \partial_x
  w\,dx\Bigr| \\
  &\le \norm{\partial_x w}_{L^\infty} \norm{\partial^p_x w}_{L^2}^2\\
  &\le C \norm{w}_{H^q} \norm{w}_{H^p}^2,
\end{align*}
since $q\ge \ell\ge 2$. Thus, 
 $$
 \frac{d}{dt} \norm{w(t)}_{H^p} \le c \alpha
 \norm{w(t)}_{H^p},
  $$
  which concludes the proof.
\end{proof}

We also need the following lemma:
\begin{lemma}
  \label{lem:burgers-2}
  If $\| u_0 \|_{H^m} \le M$ for some $m\ge 1$, then there exists
  $\overline t(M)>0$ such that $\| \Phi^t_B(u_0) \|_{H^m}  \le 2M$ for
  $0\le t \le \overline t(M)$.
\end{lemma}

\begin{proof}
  The calculation \eqref{eq:kdv4} with $m$ instead of $p$, and the
  subsequent arguments, show  that
  $$
  \norm{w(t)}_{H^m}\frac{d}{dt} \norm{w(t)}_{H^m} \le c \norm{w(t)}_{H^m}^3,
  $$
  and hence the result follows by comparing with the majorizing
  differential equation $y'=cy^2$.
\end{proof}

Finally, we will make use of the following regularity result.

\begin{lemma}
  \label{lem:burgers-3}
  If $\| u_0 \|_{H^p}\le C_0$, then there exists $\overline t$
  depending only on $C_0$, such that the solution of the inviscid
  Burgers equation with initial data $u_0$, $w(t)=\Phi^t_B(u_0)$,
  satisfies
  $$
  w\in C^2([0,\overline t], H^q) \quad\hbox{ and }\quad w\in
  C^3([0,\overline t], H^r).
  $$
\end{lemma}
\begin{proof}
  If we define, for $t\in[0,\overline t]$ with $\overline t$ from
  Lemma~\ref{lem:burgers-2} for $m=p$,
  $$
  \widetilde w(t) = u_0 +tB(u_0) + \int_0^t (t-s) \, dB(w(s))[
  B(w(s))]\, ds
  $$
  and note\footnote{We adopt the notation from \cite{AmbrosettiProdi}.} $dB(w)[B(w)]=w^2w_{xx}+2ww_x^2$, then we have $\widetilde w\in
  C^2([0,\overline t],H^q)$. Since ${\widetilde
    w}_{tt}=dB(w)[B(w)]=B(w)_t$ and ${\widetilde
    w}_{t}(0)=B(u_0)=w_t(0)$ and ${\widetilde w}(0)=u_0=w(0)$, we have
  $\widetilde w=w$.  The second statement is proved in the same way,
  by computing $\widetilde{w}_{ttt}$.
\end{proof}

\section{Local error in $H^q$}

\begin{lemma} \label{lem:local-1} The local error of the Strang
  splitting is bounded in $H^q$ by
$$
\| \Psi^\dt(u_0)-\Phi_{A+B}^\dt(u_0) \|_{H^q} \le c_1 \dt^2,
$$
where $c_1$ only depends on $\|u_0\|_{H^{p}}$.
\end{lemma}

\begin{proof} The proof follows that of the local error bound in
  \cite[\S 4.4]{Lubich:2008}. We write $e^{tA}v=\Phi^t_A(v)$ to
  indicate the linearity of the flow of $A$. We start from the
  variation-of-constants formula for
  $u(t)=\Phi^t_{A+B}(u_0)$,
  \begin{equation}
    \label{voc}
    u(t)=e^{tA}u_0 + \int_0^t e^{(t-s)A} B(u(s))\, ds,
  \end{equation}
  which is just the formula $\phi(t)-\phi(0)=\int_0^t \dot\phi(s)\,ds$ for 
  $\phi(s)=e^{(t-s)A}u(s)$, 
  and its sibling formula
  \begin{equation}
    \label{voc-b}
    B(u(s)) = B(e^{sA}u_0) + \int_0^s dB(e^{(s-\sigma) A} u(\sigma))[e^{(s-\sigma) A}B(u(\sigma))]\, d\sigma,
  \end{equation}
  which is nothing but the formula
  $B(\varphi(s))-B(\varphi(0))=\int_0^s
  dB(\varphi(\sigma))[\dot\varphi(\sigma)]\,d\sigma$ for
  $\varphi(\sigma)=e^{(s-\sigma) A} u(\sigma)$. We insert
  \eqref{voc-b} into \eqref{voc} with $t=\dt$ to obtain
$$
u(\dt)=e^{\dt A}u_0 + \int_0^\dt e^{(\dt-s)A} B(e^{sA}u_0)\, ds + e_1
$$
with
\begin{equation}
  \label{e1}
  e_1 = \int_0^\dt \int_0^s e^{(\dt-s)A}dB(e^{(s-\sigma) A} u(\sigma))[e^{(s-\sigma) A}B(u(\sigma))]\, d\sigma \, ds.
\end{equation}
On the other hand, for the result after one step of the Strang
splitting,
$$
u_1 = \Psi^\dt(u_0)= e^{\dt A/2} \Phi^\dt_B(e^{\dt A/2}u_0),
$$
we use the first-order Taylor expansion with integral remainder term
in $H^q$,
\begin{equation}
  \label{taylor}
  \Phi^\dt_B (v) = v + \dt B(v) + \dt^2 
  \int_0^1 (1-\theta) dB(\Phi^{\theta\dt}_B (v))[B(\Phi^{\theta\dt}_B (v))]\, d\theta.
\end{equation}
This is justified for $v=e^{\dt A/2}u_0\in H^p$ and for sufficiently
small $\dt$ by Lemma~\ref{lem:burgers-3}. We therefore obtain
$$
u_1= e^{\dt A} u_0 + \dt e^{\dt A/2} B(e^{\dt A/2}u_0) + e_2
$$
with
\begin{equation}
  \label{e2}
  e_2 = \dt^2 \int_0^1 (1-\theta) e^{\dt A/2}
  dB(\Phi^{\theta\dt}_B (e^{\dt A/2}u_0))[B(\Phi^{\theta\dt}_B (e^{\dt A/2}u_0))]\, d\theta.
\end{equation}
The error thus becomes
\begin{equation}
  \label{err1}
  u_1-u(\dt) = \dt\, e^{\dt A/2} B(e^{\dt A/2}u_0) - \int_0^\dt e^{(\dt-s)A} B(e^{sA}u_0)\, ds + (e_2-e_1),
\end{equation}
and hence the principal error term is just the quadrature error of the
midpoint rule applied to the integral over $[0,\dt]$ of the function
\begin{equation}
  \label{f}
  f(s) = e^{(\dt-s)A} B(e^{sA}u_0).
\end{equation}
We express the quadrature error in first-order Peano form,
$$
\dt\, f(\tfrac12\dt) -\int_0^\dt f(s) \, ds = \dt^2 \int_0^1
\kappa_1(\theta)\, f'(\theta\dt)\, d\theta,
$$
where $\kappa_1$ is the real-valued, bounded Peano kernel of the
midpoint rule. We find
$$
f'(s) = -e^{(\dt-s)A} [A,B](e^{sA}u_0)
$$
with the Lie commutator
$$
[A,B](v) = dA(v)[B(v)] - dB(v)[Av] = P(\partial_x)(vv_x) -
(P(\partial_x)v)v_x - vP(\partial_x)v_x.
$$
We have that $P$ is linear in $\partial^k_x$ for $k\le \ell$, and we
compute
\begin{align*}
   (\partial_x^\ell)(vv_x) -
(\partial_x^\ell v)v_x - v\partial_x^\ell v_x &=\sum_{k=0}^\ell
{{\ell}\choose{k}} \partial_x^{k}v\partial^{\ell+1-k}_x v -
(\partial_x^\ell v)v_x - v\partial_x^{\ell+1}v \\
&= \sum_{k=1}^{\ell-1} {{\ell}\choose{k}} \partial_x^{k}v\partial^{\ell+1-k}_x v.
\end{align*}
Hence the terms containing derivatives of order $\ell+1$ disappear, and we
obtain the commutator bound
$$
\| [A,B](v) \|_{H^{q}} \le C \bigl( \|v\|_{L^\infty} \| v \|_{H^p} +
\|v_x\|_{L^\infty} \| v \|_{H^{p}}+ \| v \|_{H^{p-1}}^2\bigr) \le C
\| v \|_{H^{p}}^2.
$$
Since $e^{tA}$ does not increase the Sobolev norms,
it follows that
$$
\| f'(s) \|_{H^q} \le C\| u_0 \|_{H^{p}}^2,
$$
and hence the quadrature error is $O(\dt^2)$ in the $H^q$ norm for
$u_0\in H^p$.  The $H^q$ norm of the remainder term $e_2-e_1$ is
bounded by $C\dt^2$ for $u_0\in H^p$ for sufficiently small $\dt$ (by
using Lemma~\ref{lem:burgers-2}). We include the details for the
convenience of the reader:
\begin{align*}
  \norm{e_1}_{H^q}&\le   \int_0^\dt \int_0^s
  \norm{e^{(\dt-s)A}dB(e^{(s-\sigma) A} u(\sigma))[e^{(s-\sigma)
      A}B(u(\sigma))]}_{H^q}\, d\sigma \, ds\\ 
  &\le \int_0^\dt \int_0^s\norm{\Big(\big(e^{(s-\sigma) A}
    u(\sigma)\big)\big(e^{(s-\sigma)
      A}B(u(\sigma))\big)\Big)_x}_{H^q}\, d\sigma \, ds\\ 
  &\le C\int_0^\dt \int_0^s\norm{u(\sigma)}_{H^{q+1}}
  \norm{B(u(\sigma))}_{H^{q+1}}\, d\sigma \, ds\\ 
  &\le C\int_0^\dt \int_0^s\norm{u(\sigma)}_{H^{q+1}}
  \norm{u(\sigma)}_{H^{q+1}}\ \norm{u(\sigma)}_{H^{q+2}}\, d\sigma \,
  ds\\ 
  &\le C \dt^2 R^3, \\
  \intertext{and} \norm{e_2}_{H^q}&\le \dt^2 \int_0^1 (1-\theta)
  \norm{e^{\dt A/2}
    dB(\Phi^{\theta\dt}_B (e^{\dt A/2}u_0))[B(\Phi^{\theta\dt}_B
    (e^{\dt A/2}u_0))]}_{H^q}\, d\theta\\ 
  &\le \dt^2 \int_0^1 \norm{\Big((\Phi^{\theta\dt}_B (e^{\dt
      A/2}u_0))(B(\Phi^{\theta\dt}_B (e^{\dt
      A/2}u_0)))\Big)_x}_{H^q}\, d\theta\\ 
  &\le\dt^2 C\int_0^1  \norm{\Phi^{\theta\dt}_B (e^{\dt
      A/2}u_0)}_{H^{q+1}} \norm{B(\Phi^{\theta\dt}_B (e^{\dt
      A/2}u_0))}_{H^{q+1}}\, d\theta\\ 
  &\le\dt^2 C\int_0^1 \norm{\Phi^{\theta\dt}_B (e^{\dt
      A/2}u_0)}_{H^{q+1}}^2
  \norm{\Phi^{\theta\dt}_B (e^{\dt A/2}u_0)}_{H^{q+2}}\, d\theta\\
  &\le C\dt^2 R^3.
\end{align*}
Thus the proof is complete.
\end{proof}

\section{Proof of Theorem~\ref{thm:one}}
The proof uses ``Lady Windermere's fan'' with error propagation by the
exact flow \cite[p.\,160, Fig.\,3.1]{HairerNorsettWanner:1993}. In our approach
the necessary regularity for estimating local errors by
Lemma~\ref{lem:local-1} is ensured by Lemma~\ref{lem:burgers-1} via an
induction argument, which we present next.

We make the induction hypothesis that for $k\le n-1$,
\begin{eqnarray*}
  &&\| u_k \|_{H^q} \le R, \quad\ 
  \| u_k \|_{H^p} \le e^{2cRk\dt}\|u_0 \|_{H^p} \le C_0 \\
  &&\| u_k - u(t_k) \|_{H^q} \le \gamma \dt,
\end{eqnarray*}
where $C_0=e^{2cRT}\|u_0 \|_{H^p}$ with $c$ from
Lemma~\ref{lem:burgers-1}, and $\gamma = K(R,T) c_1(C_0) T$ with
$K(R,T)$ from the local Lipschitz bound \eqref{lip-flow} and
$c_1(C_0)$ is the constant of Lemma~\ref{lem:local-1} for starting
values bounded by $C_0$ in $H^p$.  We then show that the above bounds
also hold for $k=n$ as long as $n\dt\le T$ and $\dt$ is sufficiently
small.

We denote, with $\Phi^t=\Phi^t_{A+B}$ for brevity,
$$
u_n^k = \Phi^{(n-k)\dt}(u_k),
$$
which is the value at time $t_n$ of the exact solution of \eqref{pde}
starting with initial data $u_k$ at time $t_k$. We note
$$
u_n = u_n^n, \quad u(t_n)=u_n^0.
$$
We estimate
\begin{eqnarray*}
  \| u_n - u(t_n) \|_{H^q} &\le& \sum_{k=0}^{n-1} \| u_n^{k+1} - u_n^k  \|_{H^q}
  \\
  &=& \sum_{k=0}^{n-1} \| \Phi^{(n-k-1)\dt}(\Psi^\dt(u_k)) - 
  \Phi^{(n-k-1)\dt}(\Phi^\dt(u_k))  \|_{H^q}.
\end{eqnarray*}
For $k\le n-2$ we have $\|\Psi^\dt(u_k)\|_{H^q} = \|u_{k+1}\|_{H^q}
\le R$, and
\begin{eqnarray*}
  \| \Phi^\dt(u_k) \|_{H^q} &\le& \| \Phi^\dt(u_k) - \Phi^\dt(u(t_k)) \|_{H^q} +
  \| \Phi^\dt(u(t_k)) \|_{H^q}
  \\
  &\le & K(R,\dt)\| u_k - u(t_k) \|_{H^q} + \| u(t_{k+1}) \|_{H^q} 
  \\
  &\le & K(R,\dt)\gamma \dt + \rho,
\end{eqnarray*}
which is bounded by $R$ if $\dt$ is so small that
$$
K(R,\dt)\gamma \dt\le R-\rho.
$$
Using \eqref{lip-flow} and Lemma~\ref{lem:local-1} we therefore have,
for $k\le n-1$ and $n\dt\le T$,
\begin{eqnarray*}
  && \| \Phi^{(n-k-1)\dt}(\Psi^\dt(u_k)) - 
  \Phi^{(n-k-1)\dt}(\Phi^\dt(u_k))  \|_{H^q} 
  \\
  && \qquad \leq
  K(R,T) \| \Psi^\dt(u_k)- \Phi^\dt(u_k)\|_{H^q} 
  \\
  && \qquad \le 
  K(R,T) c_1(C_0) \dt^2.
\end{eqnarray*}
With this estimate we obtain, again noting $n\dt\le T$,
$$
\| u_n - u(t_n) \|_{H^q} \le n K(R,T) c_1(C_0) \dt^2 \le \gamma \dt.
$$
With $\gamma \dt\le R-\rho$, we have
$$
\| u_n \|_{H^q} \le R.
$$
Since $\|\Phi^t_A(v)\|_{H^p} \le \| v \|_{H^p}$, Lemmas~\ref{lem:burgers-1}
and~\ref{lem:burgers-2} yield, for $\dt\le \overline t(R)$,
$$
\| u_n \|_{H^p} = \| \Phi^{\dt/2}_A\circ \Phi^{\dt}_B\circ
\Phi^{\dt/2}_A (u_{n-1})\|_{H^p} \le e^{2cR\dt} \|u_{n-1}\|_{H^p} \le
e^{2cRn\dt} \|u_0\|_{H^p} .
$$
This completes the proof of Theorem~\ref{thm:one}.

\section{Local error in $H^{r}$}

\begin{lemma}\label{lem:local-2}
  The local error of the Strang splitting is bounded in $H^r$ by
$$
\| \Psi^\dt(u_0)-\Phi_{A+B}^\dt(u_0) \|_{H^{r}} \le c_2 \dt^3,
$$
where $c_2$ only depends on $\|u_0\|_{H^p}$.
\end{lemma}

\begin{proof}
  The proof follows the lines of \cite[\S 5.2]{Lubich:2008}.  Instead
  of \eqref{taylor} we now use the second-order Taylor expansion
  \begin{align*}
    \Phi^\dt_B (v) &= v + \dt B(v) + \tfrac12\dt^2 dB(v)[B(v)]
    \\
    &\quad + \dt^3 \int_0^1 \tfrac12(1-\theta)^2 \Big( d^2B(\Phi^{\theta\dt}_B(v))[B(\Phi^{\theta\dt}_B(v)),B(\Phi^{\theta\dt}_B(v))] \\
    &\qquad\qquad\qquad\qquad\qquad +
    dB(\Phi^{\theta\dt}_B(v))\big[ dB(\Phi^{\theta\dt}_B(v))[B(\Phi^{\theta\dt}_B(v))]\big]\Big)\, d\theta
    \end{align*}
where henceforth we abbreviate the integral remainder term as 
$$
\dt^3 \int_0^1 \tfrac12(1-\theta)^2 
    \Big( d^2B(B,B) +dB \, dB \, B \Big)\big(\Phi^{\theta\dt}_B(v)\big) 
 \, d\theta    .
 $$
  Hence,
  \begin{align*}
    u_1&=e^{\dt A}u_0 + \dt e^{\dt A/2}B\big(e^{\dt A/2}u_0\big) +
    \tfrac{1}2\dt^2 e^{\dt A/2}dB\big(e^{\dt A/2}u_0\big)[B\big(e^{\Dt A/2}u_0\big)]  \\
    &\quad + \dt^3 \int_0^1 \tfrac12(1-\theta)^2 e^{\dt A/2}     
     \Big( d^2B(B,B) +dB \, dB \, B \Big)
\big(\Phi^{\theta\dt}_B(e^{\dt A/2}u_0)\big) \, d\theta\\
    &= e^{\dt A}u_0 + \dt e^{\dt A/2}B\big(e^{\dt A/2}u_0\big) +
    e_2,
  \end{align*}
  where $e_2$ is given by \eqref{e2}.
  
  In \eqref{e1} we express the integrand by a formula of the type
  \eqref{voc-b} by using
  \begin{equation*}
    G(u(\sigma)) = G(e^{\sigma A}u_0) + \int_0^\sigma
    dG(e^{(\sigma-\tau) A} u(\tau))[e^{(\sigma-\tau) A}B(u(\tau))]\,
    d\tau
  \end{equation*}
  for the last part of the integrand in \eqref{e1}, with
  \begin{equation*}
    G(v)=G_{s,\sigma}(v)=dB(e^{(s-\sigma)A}v)[ e^{(s-\sigma)A}\, B(v)],
  \end{equation*}
  and
  \begin{align*}
    dG(v)[w]=&\,d^2B\big(e^{(s-\sigma)A}v\big)\big[e^{(s-\sigma)A}w,e^{(s-\sigma)A}B(v)\big] 
    \\ &+ 
    dB\big(e^{(s-\sigma)A}v\big)\big[e^{(s-\sigma)A} dB(v)[w]\big].
  \end{align*}
  Then we obtain,
  \begin{align*}
    e_1&=\int_0^{\dt} \int_0^s e^{(\dt-s)A}dB\left(e^{sA}u_0\right)
    [e^{(s-\sigma)A}B\left(e^{sA}u_0\right)]\,d\sigma ds\\
    &\qquad+ \int_0^{\dt} \int_0^s \int_0^\sigma
    dG_{s,\sigma}\big(e^{(\sigma-\tau)A}u(\tau)\big)\big[e^{(\sigma-\tau)A}
    B(u(\tau))\big] \,d\tau d\sigma ds.
  \end{align*}

  We return to the error formula \eqref{err1} and write the principal
  error term
  \begin{align*}
    \dt \, e^{\dt A/2}B\big(e^{\dt A/2}u_0\big) - \int_0^\dt
    e^{(\dt-s)A} B(e^{sA}u_0)\, ds
  \end{align*}
  in second-order Peano form
  \begin{equation*}
    \dt f(\tfrac12\dt) -\int_0^\dt f(s)\,ds = \dt^3 \int_0^1
    \kappa_2(\theta)\, f''(\theta\dt)\, d\theta
  \end{equation*}
  with the second-order Peano kernel $\kappa_2$ of the midpoint rule
  and $f$ of \eqref{f}. We have
  \begin{equation*}
    f''(s) = e^{(\dt-s)A}[A,[A,B]](e^{sA}u_0).
  \end{equation*}
  The double commutator reads
  \begin{align*}
    [A,[B,A]](v)&=[A,dA B(v) - dB(v)A(v)](v)\\
    &=dA^2 B(v) - dA dB(v)A(v) - d^2A B(v) A(v) \\
    &\qquad - dA dB(v) A(v) +
    d^2B(A(v))^2 + dB(v)) dA A(v) \\
    &=A^2(B(v)) - 2A(dB(v)A(v)) + d^2B(A(v))^2 + dB(v)A^2(v)\\
    &=A^2(vv_x) - 2A(vAv)_x + ((Av)^2)_x + (vA^2v)_x,
  \end{align*}
  since $d^2A=0$.  The operator $A$ is linear in $\partial^\ell_x$,
  and the highest order derivative in the term $((Av)^2)_x$ is
  $\ell+1\le 2\ell-1$, so we consider
  \begin{align*}
    \partial^{2\ell}_x(vv_x) &- 2\partial^{\ell+1}_x\left(v\pl
      v\right)+
    (v\partial^{2\ell}_xv)_x \\
    &= \sum_{j=0}^{2\ell}
    {{2\ell}\choose{j}} \partial^{2\ell-j}_xv\partial^{j+1}_x v -
    2\sum_{j=0}^{\ell+1} {{\ell+1}\choose{j}} \partial^{\ell+1-j}_x
    v \partial^{\ell+j}_x v + v_x \partial^{2\ell}_x v +
    v\partial^{2\ell+1}v.
  \end{align*}
  We note that in the above sum, the terms containing derivatives of
  order $2\ell+1$ and $2\ell$ disappear, so
  \begin{equation*}
    \| f''(s) \|_{H^r} \le C \| u_0 \|_{H^p}^2.
  \end{equation*}
  Now for the difference of \eqref{e1} and \eqref{e2},
  \begin{align*}
    e_2 -e_1 &= \tfrac12 \dt^2g(\tfrac12\dt,\tfrac12\dt) -
    \int_0^\dt\int_0^s g(s,\sigma)\,d\sigma\, ds + \tilde e_2 -
    \tilde e_1,
    \intertext{where}
      g(s,\sigma) &= e^{(\dt-s)A} dB( e^{s A}u_0)\, e^{(s-\sigma)A}B(e^{\sigma A}u_0)
    \intertext{and}
    \tilde{e}_1 &= \int_0^{\dt} \int_0^s \int_0^\sigma
    dG_{s,\sigma}\left(e^{(\sigma-\tau)A}u(\tau)\right) e^{(\sigma-\tau)A}
    B(u(\tau)) \,d\tau d\sigma ds, \intertext{and} \tilde{e}_2 &=
    \dt^3 \int_0^1 \tfrac12(1-\theta)^2 e^{\dt A/2} \Bigl( d^2B(B,B) +
    dB \,dB \, B\Bigr)(\Phi^{\theta\dt}_B(u_0))\, d\theta.
  \end{align*}
  To estimate the remainder terms $\tilde e_i$ we calculate
  \begin{align*}
    \norm{dG_{s,\sigma}(v)w}_{H^r}&\le
    \norm{d^2B\left(e^{(s-\sigma)A}B(v),w\right)}_{H^r}+
    \norm{dB\left(e^{(s-\sigma)A} v
      \right)e^{(s-\sigma)A}dB(v)w}_{H^r}\\
    &\le
    C\left(\norm{B(v)}_{H^{r+1}}\norm{w}_{H^{r+1}}+\norm{v}_{H^{r+1}}
      \norm{dB(v)w}_{H^{r+1}}\right)\\
    &\le C\left(\norm{v}_{H^{r+2}}^2 \norm{w}_{H^{r+1}} +
      \norm{v}_{H^{r+1}}\norm{v}_{H^{r+2}} \norm{w}_{H^{r+2}}\right)\\
    &\le C\norm{v}_{H^{r+2}}^2\norm{w}_{H^{r+2}},
  \end{align*}
  and
  \begin{align*}
    \norm{\left(d^2B(B,B) + dB\, dB\, B\right)(v)}_{H^r}&\le
    \norm{d^2B(B(v),B(v))}_{H^r}+ \norm{dB(v)dB(v)B(v)}_{H^r}\\
    &\le C\left(\norm{B(v)}_{H^{r+1}}^2 +
      \norm{v}_{H^{r+1}}\norm{dB(v)B(v)}_{H^{r+1}}\right) \\
    &\le C\left(\norm{v}_{H^{r+2}}^4 +
      \norm{v}_{H^{r+2}}^2\norm{B(v)}_{H^{r+2}}\right) \\
    &\le C \norm{v}_{H^{r+3}}^4.
  \end{align*}
  Then, using Lemma~\ref{lem:burgers-1}, the
  remainder terms are bounded by
  \begin{equation*}
    \norm{ \tilde e_1}_{H^r} + \norm{\tilde e_2}_{H^r} \le C \dt^3
    \big( \|u_0\|_{H^p}^4+\|u_0\|_{H^p}^3\big),
  \end{equation*}
  since $2\ell -1 \ge 3$.
  
  The first two terms in $e_2-e_1$ are the quadrature error of a
  first-order two-dimensional quadrature formula, which is bounded by
  \begin{align*}
    & \Bigl\| \tfrac12 \dt^2g(\tfrac12\dt,\tfrac12\dt) -
    \int_0^\dt\int_0^s g(s,\sigma)\,d\sigma\, ds \Bigr\|_{H^r}
    \\
    &\qquad\qquad \le C \dt^3 \bigl( \max \| \partial g/\partial
    s\|_{H^r} + \max \| \partial g/\partial\sigma \|_{H^r}\bigr),
  \end{align*}
  where the maxima are taken over the triangle
  $\{(s,\sigma):0\le\sigma\le s \le \dt\}$.  In order to estimate the
  partial derivatives we write
  \begin{equation*}
    g(s,\sigma)=e^{(\dt-s)A} dB(v(s)) w(s,\sigma),
  \end{equation*}
  where
  \begin{equation*}
    v(s)=e^{sA}u_0\ \text{ and }\ w(s,\sigma)=e^{(s-\sigma)A}B(v(\sigma)).
  \end{equation*}
  With this notation
  \begin{align*}
    \frac{\partial g}{\partial s}&= e^{(\dt-s)A}
    \left(-AdB(v(s))w(s,\sigma)) +d^2B(Av(s),w(s,\sigma))+
      dB(v(s))Aw(s,\sigma)\right)\\
    &= e^{(\dt-s)A}\left(-A(v(s)w(s,\sigma))+Av(s) w(s,\sigma)+
      v(s)Aw(s,\sigma)\right)_x.
  \end{align*}
  Since the flow determined by $A$ does not increase the Sobolev
  norms, it suffices to estimate the $H^{r+1}$ norm of
  $-A(vw)+(Av)w+vAw$. We see that the derivatives of order $\ell$
  vanish, so that
  \begin{equation*}
    \norm{ \frac{\partial g}{\partial s}}_{H^{r}} \le
    C\norm{v(s)}_{H^{r+\ell}} \norm{w(s,\sigma)}_{H^{r+\ell}} \le
    C\norm{u_0}_{H^{r+\ell}} \norm{u_0}_{H^{r+\ell+1}}^2\le C\norm{u_0}_{H^{p}}^3,
  \end{equation*}
  since $\ell\ge 2$. For the other partial derivative, we get
  \begin{align*}
    \frac{\partial g}{\partial \sigma} &= e^{(\dt-s)A}
    dB(v(s))\left(e^{(s-\sigma)A}\left(-AB(v(s))
        +dB(v(\sigma))Av(\sigma)\right)\right),
  \end{align*}
  so that
  \begin{equation*}
    \norm{ \frac{\partial g}{\partial \sigma} }_{H^r}\le C
    \norm{v(s)}_{H^{r+1}} \norm{-AB(v(s))
      +dB(v(\sigma))Av(\sigma)}_{H^{r+1}}
  \end{equation*}
  The last factor is $-A(vv_x) + (vAv)_x$, and again the derivatives
  of order $\ell+1$ vanish, so that
  \begin{equation*}
    \norm{ \frac{\partial g}{\partial \sigma} }_{H^r}\le C \norm{u_0}_{H^p}^3.
  \end{equation*}
  Therefore we obtain
  \begin{equation*}
    \norm{e_2 - e_1}_{H^r} \le C \dt^3\big( \|u_0\|_{H^p}^4+ \|u_0\|_{H^p}^3\big),
  \end{equation*}
  which together with the bound for the quadrature error of the
  midpoint rule for $f$ yields the stated result.
\end{proof}

\section{Proof of Theorem~\ref{thm:two}}

We use ``Lady Windermere's fan'' in $H^{r}$ in combination with the $H^p$
bound of $u_n$ from Theorem~\ref{thm:one} and the local error bound
from Lemma~\ref{lem:local-2}.  For $t_n=n\dt\le T$ we have, as in the
proof of Theorem~\ref{thm:one},
\begin{eqnarray*}
  \| u_n - u(t_n) \|_{H^r} &\le&  \sum_{k=0}^{n-1} \| \Phi^{(n-k-1)\dt}(\Psi^\dt(u_k)) - 
  \Phi^{(n-k-1)\dt}(\Phi^\dt(u_k))  \|_{H^r}
  \\
  &\le& \sum_{k=0}^{n-1} K(R,T) \| \Psi^\dt(u_k)-\Phi^\dt(u_k) \|_{H^{r}}
  \\
  &\le& \sum_{k=0}^{n-1} K(R,T) c_2(C_0) \dt^3 \le K(R,T) c_2(C_0) T \dt^2.
\end{eqnarray*}
This completes the proof of Theorem~\ref{thm:two}.


\end{document}